\documentclass[12pt,a4paper]{article}  
\usepackage{hyperref}
\usepackage{t1enc}
\usepackage[latin1]{inputenc}
\usepackage[english]{babel}
\usepackage{graphicx}
\usepackage{amsmath}
\usepackage{amssymb}
\usepackage{amsthm}
\font\bBB=msbm10

\def\bBR{\mbox{\bBB R}}

\def\R3{\bBR ^3}
\begin{document}
\title {From Foucault's Pendulum to the Gauss--Bonnet Theorem}
\author{
Orlin Stoytchev\thanks{American University in Bulgaria, 2700 Blagoevgrad, Bulgaria} }
\date{}
\maketitle
\abstract
{We present a self-contained proof of the Gauss-Bonnet theorem for two-dimensional surfaces embedded in $\R3$ using just classical vector calculus. The exposition should be accessible to advanced undergraduate and non-expert graduate students. It may be viewed as an illustration and exercise in multivariate calculus and a motivation to go deeper into the fields of geometry and topology.}
\section{Introduction}
\par\smallskip
The Gauss--Bonnet theorem states that the total curvature of a closed two-dimensional oriented surface (i.e., the integral of the Gaussian curvature over that surface) is equal to the Euler characteristic of the surface multiplied by $2\pi$. This is a beautiful result relating a geometric quantity -- the curvature -- to a purely topological one -- the Euler characteristic. Given the right intuition about geodesics and parallel transport, one can prove the Gauss-Bonnet theorem for embedded surfaces with little more than vector calculus and definitely without heavy differential-geometric machinery.\par
When L\'eon Foucault built his famous pendulum in 1851 in Paris (first in the Paris Observatory, moved a little later to the Panth\'eon), he hardly had in mind any deep connections with geometry and topology. His aim was, of course, to demonstrate by a direct physical experiment the rotation of the earth about its axis. When the earth makes one full rotation relative to the stars (which happens in approximately 23 hours and 56 minutes and is called sidereal day), the plane in which the pendulum in Paris is swinging, rotates relative to the ground by $271.1^{\circ}$ clockwise. A pendulum at the north pole would rotate by exactly $360^{\circ}$ while at the equator there will be no rotation. In general the angle of rotation is given by $360^{\circ} \sin\phi$, where $\phi$ is the geographic latitude. The explanation behind this formula is that the earth is curved and the normal vector to the surface at some point traces a cone as this point traces a circle on the sphere. At the same time the tangent vector giving the direction of swinging of the pendulum undergoes a {\it parallel transport} (to be defined later) along the circle as there are no forces to cause any rotation around the normal vector. In a sense Foucault's pendulum tells us not only that the earth rotates, but that it is curved (in case we knew Gauss-Bonnet's theorem but weren't sure of the earth's shape).\par
The same phenomenon can be viewed slightly differently. Let $C$ denote the circle coinciding with one of the geographic parallels on the sphere, at latitude $\phi$, with counterclockwise orientation. When one performs a parallel transport along $C$ of a tangent vector, the vector in general will be rotating relative to $C$, since $C$ is not a geodesic (a large circle), unless it is the equator. The vector will fail to return to its original orientation when coming to the initial point. The angle between the initial and the final vectors  is sometimes  called {\it deficit angle} (or {\it angular deficit} or {\it angular defect}). For the case at hand the calculation gives $2\pi(1-\sin\phi)$. If one calculates the area on the sphere, bounded by $C$, one finds out that it is given by  $2\pi(1-\sin\phi) R^2$, where $R$ is the radius of the sphere. This is a very special case of the Gauss-Bonnet formula and the behavior of Foucault's pendulum demonstrates its validity.\par
In the next section we will define in a simple and rather intuitive way the notions of a geodesic and parallel transport. We will derive the formula for the deficit angle along the curve $C$ as above. Then, using appropriate technique for calculation, we will obtain the Gauss-Bonnet formula for an arbitrary closed simple curve on the sphere. Not surprisingly this derivation invokes Stokes' theorem. The general case -- arbitrary closed oriented two-dimensional surface embedded in  $\bBR^3$ is treated in the last section. The key is to consider carefully the Gauss map given by the unit normal vector to the surface, which sends each point on the surface to a point on the unit sphere. It turns out that the general case is reduced to the case of the sphere by a simple change of variables formula. The Gauss curvature on the surface appears in this setting as the Jacobian of the Gauss map. \par
\section{Geodesics, parallel transport, flat and curved surfaces}\par\smallskip
Consider a smooth two-dimensional surface $\sigma$ embedded in $\R3$. Intuitively a geodesic curve on $\sigma$ is a smooth curve $C\subset\sigma$, such that if you travel along it with constant speed, at any given point there will be no component of the acceleration in the tangent plane to the surface. If we imagine rolling a ball on the surface (e.g., some adhesive force causes the ball to stick to the surface but does not restrict it in any other way) and we do this in weightlessness, the ball will trace precisely a geodesic. If there was no restriction, the geodesic would be a straight line in $\R3$. The condition to stay on the surface makes the trajectory  curved  in general, but in such a way that the acceleration stays normal to the surface, since the only force is the normal adhesive force. Suppose $C$ is a smooth curve on $\sigma$ and let ${\bf c}(t)$ be a smooth parametrization of $C$, such that the corresponding velocity vector ${\bf v}(t):={d{\bf c}\over dt}$ has constant norm. (In fact if we parametrize $C$ by arc-length, the norm of ${\bf v}(t)$ will be one.) Let ${\bf a}(t):={d{\bf v}\over dt}$ be the acceleration and denote by ${\bf a}_T(t)$ its projection in the tangent plane (at the point ${\bf c}(t)$). So by definition $C$ is called a {\it geodesic} if  ${\bf a}_T(t)=0,\,\,\forall t$. If we consider as an example a circle on a sphere and imagine a point moving with constant speed along this circle, the acceleration is obviously a vector pointing towards the center of the circle. Unless  the circle is a large circle (like the equator) its center does not coincide with the center of the sphere and the acceleration has nonzero tangential component. In fact the only geodesics on the sphere are (parts of) large circles.\par
Defining parallel transport along a geodesic is simple. If you have a family of tangent  vectors ${\mathbf w}(t)$ defined for each point ${\bf c}(t)$ of the geodesic  $C$ and  $\|{\mathbf w}(t)\| = \mbox{const.}$, then we say that this family of vectors has been obtained by parallel transport along $C$ of one vector, say ${\mathbf w}(0)$ if the angle between ${\mathbf w}(t)$ and the geodesic stays constant. More precisely we should have constant angle between ${\mathbf w}(t)$ and ${\mathbf v}(t)$. Obviously the velocity vectors  ${\mathbf v}(t)$ along a geodesic $C$ constitute the simplest example of parallel transport. When the curve $C$ is not a geodesic we may use the intuition coming from Foucault's pendulum. Let ${\bf w}(0)$ be a unit tangent vector at the initial point ${\bf c}(0)$ giving the direction of swinging of the pendulum. When we move the latter along $C$, there will be no rotation of the pendulum around the normal vector. In this way we obtain a unit vector ${\bf w}(t)$ for $t$ and we should say that this has been obtained from ${\bf w}(0)$ by parallel transport along $C$. So the condition is similar to the one we imposed for the velocity ${\bf v}(t)$ when we were defining geodesics, i.e. $({d{\bf w}\over dt})_{_T}=0$, or in words, the rate of change of ${\bf w}(t)$ has no tangential component. Now, because $C$ is no longer a geodesic, the angle between ${\bf w}(t)$ and $C$ will be changing. The rate of change of this angle, denoted further by $\omega(t)$, has magnitude equal to the rate of rotation of  ${\bf v}(t)$ in the tangent plane and opposite sign. Therefore we have:
$$|\omega(t)|={\|{\bf a}_T(t)\| \over \|{\bf v}(t)\|}\quad .$$
Let's take the closed curve $C$ to be the geographic parallel at geographic latitude $\phi$ (with counterclockwise direction) and   calculate the total angle of rotation (relative to $C$) of a vector ${\bf w}$ transported parallel to itself. We can parametrize $C$ with the path ${\bf c}(t)=\cos 2\pi t \cos\phi\  {\bf i}+\sin 2\pi t \cos\phi\  {\bf j}+\sin \phi\  {\bf k}$, $t\in [0, 1]$. Straightforward calculation gives ${\bf a}(t)=-4\pi^2\cos\phi(\cos 2\pi t\  {\bf i}+\sin 2\pi t\  {\bf j})$ and ${\bf a}_T(t)=-4\pi^2\cos\phi (\cos 2\pi t\ \sin^2\phi\ {\bf i}+\sin 2\pi t\ \sin^2\phi\ {\bf j}-\sin\phi\cos\phi\  {\bf k})$. Finally we obtain $|\omega(t)|=2\pi|\sin\phi|$ and therefore the total angle of rotation of ${\bf w}$ when coming to the initial point, which is obtained by integrating $\omega(t)$ between $0$ and $1$ (and figuring the correct sign) is equal to
$-2\pi\sin\phi\ .$
This implies that the angle between ${\bf w}(0)$ and ${\bf w}(1)$, i.e. the  deficit angle, is given by 
\begin{equation}\label{sphere}
\Omega=2\pi(1-\sin\phi)\ .
\end{equation}
\par
The same result can be obtained (see, e.g. \cite{Arnold}) using a simple geometric approach by considering the circular cone tangential to the sphere along the curve $C$ (Fig. \ref{conical}). (When $C$ is the equator the cone degenerates into a cylinder, i.e. a cone with vertex (apex) at infinity.) Since the notions of geodesics and parallel transport along $C$ depend only on the tangent planes along $C$ and these are common for both surfaces, a curve will be a geodesic on the sphere if and only if it is a geodesic on the cone and parallel transport gives the same result for both surfaces. This is a general property for any two surfaces touching each other along a curve. Suppose now that we cut the cone along some line, from the boundary to the apex and lay it flat on the plane. This process is called "developing" the surface and surfaces allowing this are called {\it developable}. We will give a definition of the latter term but intuitively these are surfaces which can be obtained by gluing patches cut from sheets of paper.
\begin{figure}[h]
\centering
\includegraphics[width=60mm]{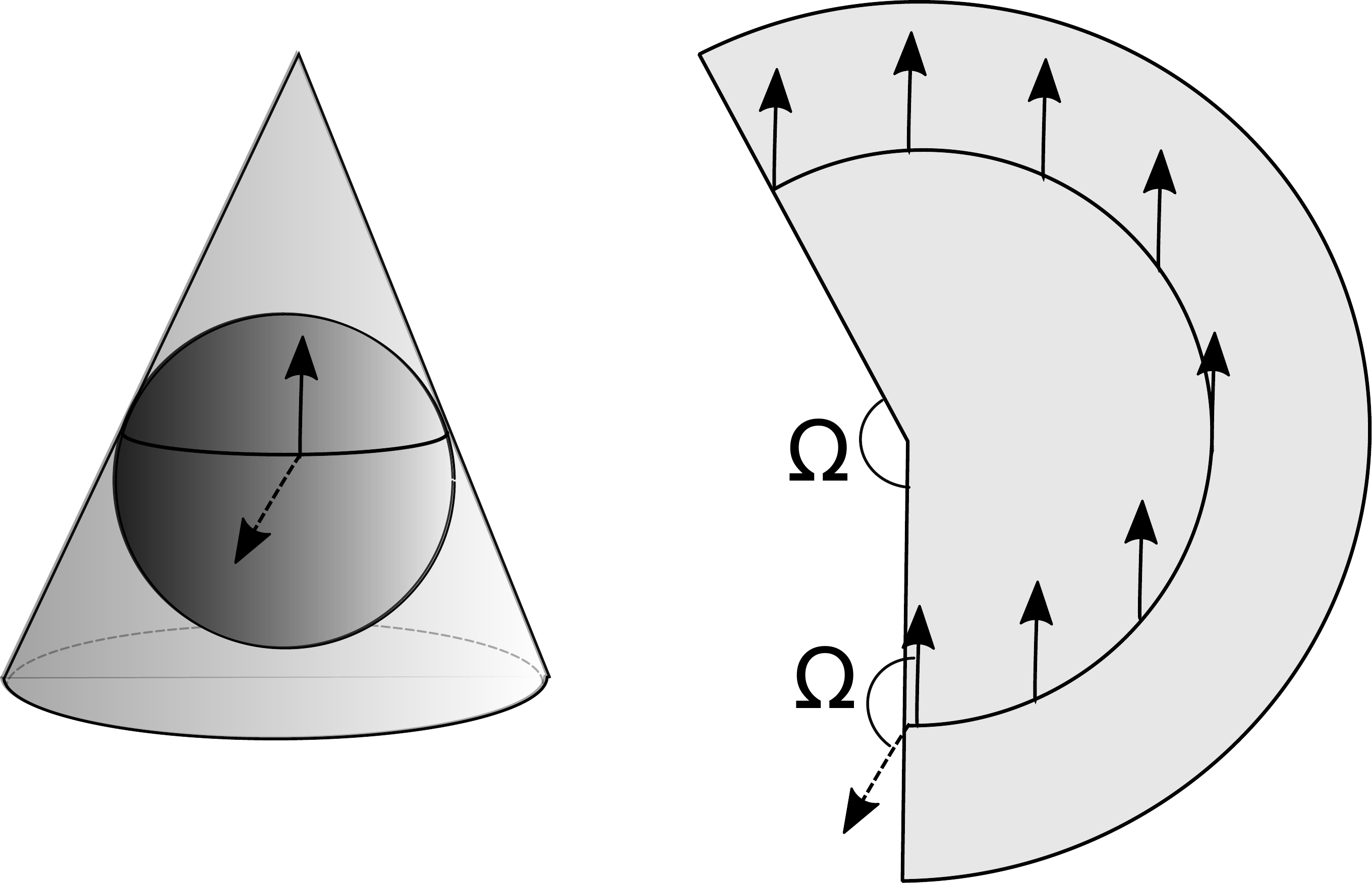}
\caption{Parallel transport along a circle on the sphere, using a tangential cone\label{conical}}
\end{figure}
It is clear that the process of developing a surface does not influence parallel transport - a spacial curve $C$ transforms into a plane curve $C'$, the normal component of ${\bf a}(t)$ along $C$ turns into zero under this transformation while the tangential component of ${\bf a}(t)$ remains unchanged. The developed cone becomes a disc with a cut-out sector with angle $\Omega$, hence the name deficit angle. Parallel transport of a vector ${\bf w}$ along the curve $C'$ is the usual parallel transport in the plane. The angle between ${\bf w}(0)$ and  ${\bf w}(1)$ is $\Omega$. A relatively simple exercise in elementary geometry allows us to calculate this angle and obtain the result stated in Equation (1).\par
The idea of the deficit angle has far-reaching ramifications. A developable smooth surface like the cone with its apex removed is very similar to the plane in that it can be laid flat on the plane and if you transport a vector parallel to itself along a contractible curve $C$ on this surface it comes back to itself. We call such surfaces Gaussian-flat, or intrinsically flat. By contrast, no piece of the sphere can be laid flat on the plane and if you perform a parallel transport of a vector along any (non-constant) simple closed curve on the sphere, the returning vector will generally have different direction from the initial. We say that such surfaces have nonzero Gaussian curvature. The precise definition of Gaussian curvature will be given later but it is some local quantity $K$ defined at each point of a smooth surface and measuring the deviation of the surface from being developable. For the sphere it is clear that $K$ should be the same at each point since the neighborhood of every point looks exactly the same as the neighborhood of every other. The essential part of the proof of the Gauss-Bonnet theorem is to show that if $C$ is a simple, positively oriented and contractible curve on a surface, then the deficit angle, i.e. the angle between the initial vector ${\bf w}(0)$ and the final vector ${\bf w}(1)$ is given by the integral of $K$ over the part of the surface surrounded by $C$. In other words we have
\begin{equation}\label{Eq2}
\Omega=\int\hskip-6pt\int_{\sigma} K\,dA\ ,
\end{equation}
$\sigma$ being the part of the surface surrounded by $C$. This formula, by the way, immediately shows that for a sphere with radius $R$ we must have $K\propto {1\over R^2}$. Coming back to our argument when calculating the deficit angle along the geographic parallel $C$ we may view replacing the upper part of the sphere by the cone touching it along $C$ as the limit of a process where we replace the upper part of the sphere by a conical surface with its apex cut and capped by a smaller and smaller spherical part. Since the conical part is Gaussian-flat the curvature is concentrated at the cap and in the limit it becomes infinite but in such a way that the integral remains equal to $\Omega$. In other words we may view our cone as a curved surface if we say that the curvature is zero everywhere except at the apex where it has a $\delta$-function-like singularity.\par
Take an arbitrary closed surface $S$ and choose a triangulation of $S$, i.e. cover it by curvilinear triangles. Let $V$, $E$ and $F$ denote the number of vertices, edges and triangles, respectively. Then the Euler characteristic $\chi (S):= V-E+T$ is independent of the choice of triangulation and is a topological invariant. Now replace the curved edges by straight ones and the curvilinear triangles by flat triangles. We get a polyhedron $S'$ having the same Euler characteristic. If we calculate the deficit angle at one of the vertices we have
$$\Omega_i=2\pi - \sum_{k=1}^{n_i}\gamma_{ik}\ ,$$
where $n_i$ is the number of triangles meeting at the vertex $v_i$. Note that the sum of the angles at some vertex may exceed $2\pi$ and the deficit angle will be negative in this case. The corresponding surface near this vertex will look like an (uncomfortable edgy) saddle, possibly with multiple "ridges" and "troughs". Summing over all vertices we get the total deficit angle:
$$\Omega=\sum_{i=1}^V \Omega_i=2\pi V-\sum_{i=1}^V \sum_{k=1}^{n_i}\gamma_{ik}\ .$$
The double sum on the right is in fact the sum of all angles of all triangles of our polyhedron and therefore we can write
\begin{equation}\label{Eq3}
\Omega=2\pi(V-{F\over 2})=2\pi (V-E+F)=2\pi\chi(S')\ .
\end{equation}
(We use the fact that each triangle contains three edges and each edge is common to two triangles and thus $E=3F/2$.) The result in Equation (3) is the content of Euler's theorem, namely that the total deficit angle for an arbitrary polyhedron is $2\pi$ times the Euler characteristic.\par\noindent
{\bf Note:} We proved Equation (3) for triangles while the faces of a polyhedron are usually arbitrary polygons and the number $F$ in the definition of the Euler characteristic is the number of faces. This, however does not cause any difficulty as each polygon can be broken down into triangles by adding some edges. As the reader may easily check this process does not change the Euler characteristic.\par
Thus the theorem of Euler about polygons may be viewed as a discrete analog of Gauss-Bonnet's theorem if we think of the total deficit angle as the integral of the curvature over the surface of the polyhedron and the curvature is "concentrated" at the vertices.\par
Let now $C$ be an arbitrary simple (i.e., having no self-intersections) smooth closed curve on the unit sphere. We aim to demonstrate that the formula in Equation (2) is still valid in this case. First we perform the proof for a {\it small} curve which does not leave some open half-sphere. We choose a point surrounded by $C$ and call it the``north pole'' of the sphere. We approximate $C$ by a piecewise smooth curve $C'$ consisting of pieces of meridians (which are geodesics) and pieces of parallels (which are not geodesics). Since during parallel transport along a geodesic a vector $\mathbf w$ preserves its angle relative to it, it is clear that we must sum the contributions to the deficit angle $\Omega'$ from the motion along the pieces of parallels. From our previous calculation we have that the contribution along a piece of a parallel at a geographic latitude $\phi$, corresponding to azimuthal change  $\Delta\theta$ is equal to $(1-\sin\phi)\Delta\theta=(\cos\phi)^{-1}(1-\sin\phi)\Delta s$, where $\Delta s$ is the corresponding arc length taken with plus or minus sign depending on the sign of $\Delta\theta$. In other words we have
$$\Omega'=\sum_{i=1}^N(1-\sin\phi_i)\Delta\theta_i=\sum_{i=1}^N {1-\sin\phi_i \over \cos\phi_i}\Delta s_i\ ,$$
which is a Riemann sum for the line integral of a suitable vector field ${\bf F}$ along $C$. We need a vector field which picks only the parts of $C'$ along parallels, so ${\bf F}$ must be along parallels. A unit vector field in $\R3$ with this property is given by ${1\over r}(-y{\bf i}+x{\bf j})$, where $r=\sqrt {x^2+y^2}$. Taking into account that on the sphere $\cos\phi=r=\sqrt{ x^2+y^2}$, we see that a vector field that does the job is
$$\mbox{\bf F}={1-\sqrt{1-x^2-y^2} \over x^2+y^2}(-y{\bf i}+x{\bf j})\ .$$
Therefore we obtain, applying Stokes' theorem and denoting by $\sigma$ the surface surrounded by $C$:
\begin{equation}\label{Eq4}
\Omega=\int_C \mbox{\bf F}\cdot d\mbox{\bf s}=
\int\hskip-6pt\int_{\sigma} \mbox{curl \bf F}\cdot d\mbox{\bf A}=
\int\hskip-6pt\int_{\sigma} \mbox{curl \bf F}\cdot \mbox{\bf N}\ dA\ .
\end{equation}
A routine calculation shows that
$$ \mbox{curl \bf F}={1\over \sqrt{1-x^2-y^2}}\ {\bf k}$$
and for a point on the sphere with spherical coordinates $(\theta,\phi)$ we can write
$$ \mbox{curl \bf F}\cdot {\bf N}={1\over\sin\phi}{\bf k}\cdot {\bf N}=1\ .$$
Thus Equation (4) reduces to
\begin{equation}\label{Eq5}
\Omega=
\int\hskip-6pt\int_{\sigma} dA\ ,
\end{equation}
i.e., the deficit angle on the sphere is given by the area of the surface surrounded by $C$.\par\noindent
{\bf Remark 1} We chose to present the somewhat clumsy derivation above, since it assumes just familiarity with classical vector calculus. The same result can be obtained using differential forms. Namely, if we introduce the one-form, which in spherical coordinates is given by $\alpha=(1-\sin\phi)d\theta$, it is easy to see that
$$\Omega=\int_C\alpha\ .$$
Applying the generalized Stokes' theorem for forms, we have
$$\Omega=\int_C\alpha=\int\hskip-6pt\int_{\sigma}d\alpha=
\int\hskip-6pt\int_{\sigma}-\cos\phi\, d\phi \wedge d\theta=
\int\hskip-6pt\int_{\sigma}\cos\phi\, d\theta \wedge d\phi=
\int\hskip-6pt\int_{\sigma}\cos\phi \,d\theta d\phi=
\int\hskip-6pt\int_{\sigma}dA\ .$$
(The area element on the sphere in spherical coordinates  is, of course, $dA=\cos\phi\,d\theta d\phi$. Strictly speaking, the one-form $\alpha$ may seem not to be defined at the north pole and indeed the spherical coordinates
$(\theta,\phi)$ don't provide a local chart, but notice that the form becomes 0 at this point and thus $\alpha$ is in fact well-defined.)\par\noindent
{\bf Remark 2} The same conclusion remains valid if we allow $C$ to be piecewise-smooth curve, i.e., a curvilinear polygon. It is obvious how to do parallel transport of ${\bf w}$ across a vertex of the polygon - the angle between ${\bf w}$ and $C$ jumps to a new value, the change being equal to minus the angle between the positive directions of $C$ before and after the vertex.\par\noindent
{\bf Remark 3} When calculating the deficit angle $\Omega'$ along the approximating curve $C'$ we did not take into account the jumps of the angle between a parallel vector and the separate pieces of $C'$ that happen at the corners. These jumps are by $\pm\pi/2$. The point is that $C'$ is closed and turns around the north pole once, so there must be an equal number of ``left turns'' and ``right turns''.\par
Having proved the Gauss-Bonnet formula on the sphere for small curves we now extend the result for arbitrary simple closed curves $C$ by a standard technique. If $C$ is not small we break the region surrounded by it into two smaller regions by introducing an auxiliary open curve between two points on $C$. The curve $C$ becomes a concatenation of two ``smaller'' curves $C_1$ and $C_2$ sharing a common boundary but with opposite orientations. The deficit angle, which reverses sign when switching the orientation, will be the sum of the deficit angles along $C_1$ and $C_2$ and at the same time the area surrounded by $C$ is the sum of the areas surrounded by $C_1$ and $C_2$. Thus in a finite number of steps we reduce the general case to the one for small curves.\par
\section{Gaussian curvature. Gauss-Bonnet theorem for arbitrary closed surfaces}\par\smallskip
Let now $S$ be an arbitrary oriented smooth surface embedded in $\R3$. The {\it Gauss map} $G:S\rightarrow S^2$ is defined as follows - for each point on $S$ take the unit normal vector to the surface at this point and identify the latter with the corresponding point on the unit sphere $S^2$. Clearly $G$ is a smooth map which is not one-to-one. A contractible simple closed curve $C\subset S$  will be mapped to a  closed curve $C'\subset S^2$ which can have self-intersections or even degenerate to a point. (See Fig. \ref{torus} for an illustration.) The map $G$ does not preserve in general the orientation of a curve. In fact, as Fig. \ref{torus} suggests, a positively oriented closed curve on the interior half of the torus, where the curvature is negative, is mapped by $G$ to a negatively oriented curve on $S^2$.
\begin{figure}[h]
\centering
\includegraphics[width=80mm]{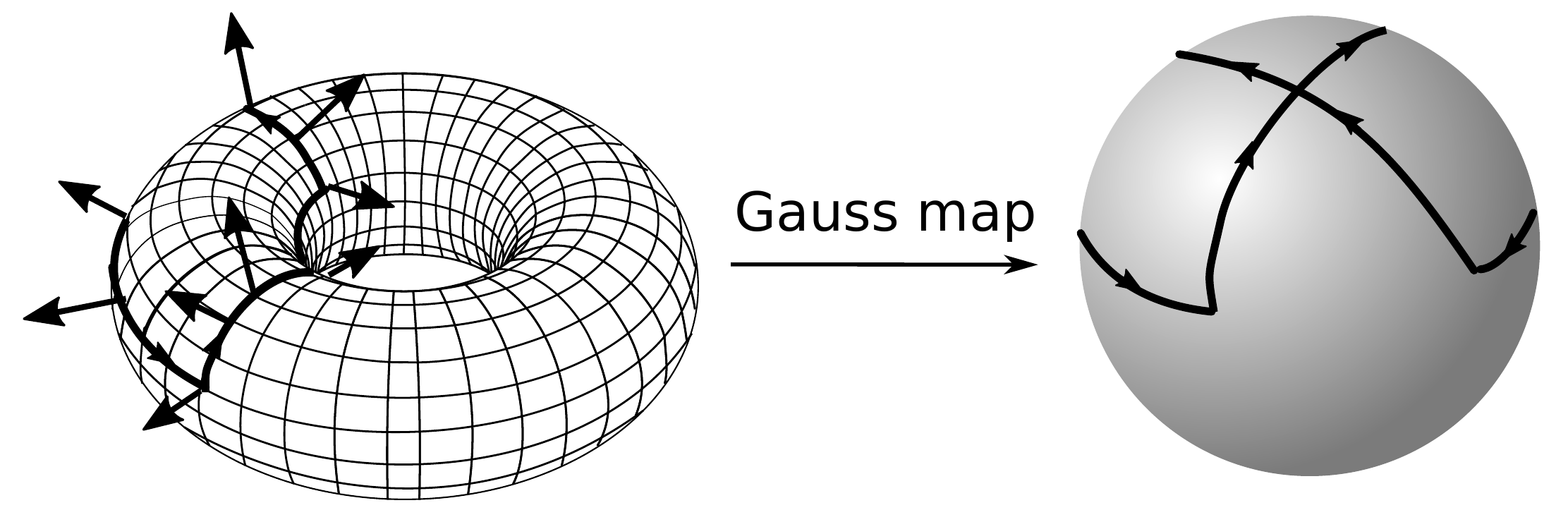}
\caption{The effect of the Gauss map for a curve on the torus\label{torus}}
\end{figure} 
\par
The main result in this section will be a proof of Equation (\ref{Eq2}) with $K$ being the Jacobian of the Gauss map $G$. As a first step we show the following\par\smallskip\noindent
\newtheorem{prop}{Proposition}
\begin{prop}\label{Prop1}
 With the same notations as above, the deficit angle $\Omega$ along $C$ is equal to the deficit angle $\Omega'$ along $C'$.
\end{prop}
\begin{proof}
 Let ${\bf c}(t)$, $t\in [0,1]$ be a parametrization of $C$. We will approximate the surface $S$ in a neighborhood of $C$ by a Gaussian-flat strip in the following way: Divide $[0,1]$ into $n$ equal subintervals and let $t_i$ be the endpoint of the {\it i}\,th interval. Denote by ${\bf N}_i$ the unit normal vector at the point ${\bf c}(t_i)$. Using the spherical angles $\theta$ and  $\phi$ we have ${\bf N}_i=\cos \theta_i \cos\phi_i\  {\bf i}+\sin \theta_i \cos\phi_i\  {\bf j}+\sin \phi_i\  {\bf k}$. (Note that we use the less common definition for $\phi$ as the angle between  the horizontal plane and the vector. The so-called polar angle, i.e. the angle between the $z$-axis and the vector, is $\pi - \phi$.) At each point ${\bf c}(t_i)$ take (a rectangular piece of) the tangent plane $P_i$. 
\begin{figure}[h]
\centering
\includegraphics[width=90mm]{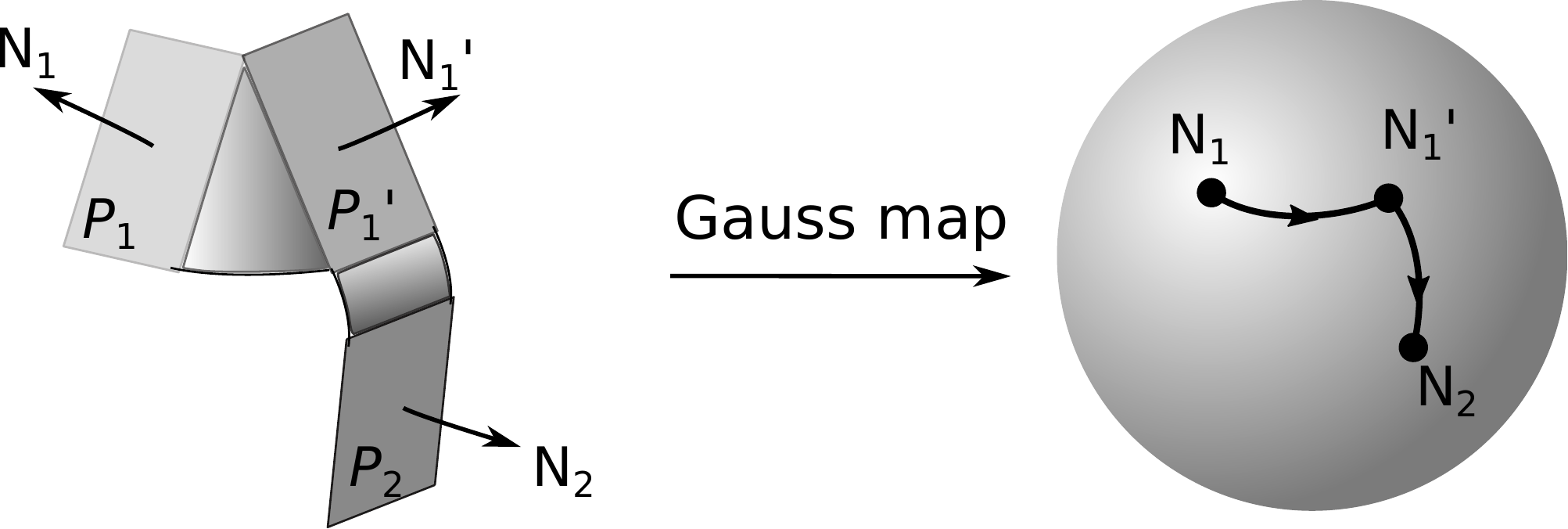}
\caption{Approximating the surface $S$ in a neighborhood of $C$ by a Gaussian-flat strip \label{GM2}}
\end{figure}
Consider now two consecutive planes $P_i$ and $P_{i+1}$. If $\phi_i=\phi_{i+1}$ we connect them by a conical surface defined in the obvious way. In the generic case when $\phi_i\ne\phi_{i+1}$ we take an auxiliary plane $P_i'$, perpendicular to the auxiliary vector ${\bf N}_i'=\cos \theta_{i+1} \cos\phi_i\  {\bf i}+\sin \theta_{i+1} \cos\phi_i\  {\bf j}+\sin \phi_i\  {\bf k}$ and then connect the latter to $P_{i+1}$ using a cylindrical surface  (see Fig. \ref{GM2}). (The auxiliary plane $P_i'$ is not necessarily tangent to $S$.) We close the strip by connecting $P_n$ to $P_1$. In this way we obtain a differentiable surface - the strip $S'$, which will be developable and will be tangent to $S$ at the points ${\bf c}(t_i)$. The smooth closed curve $C\subset S$ can be approximated by a closed piecewise-smooth curve $C_n\subset S'$, e.g., by taking a constant-$\phi$ curve from ${\bf c}(t_i)$ to the auxiliary point ${\bf c}'(t_i)$ and then a constant-$\theta$ curve from ${\bf c}'(t_i)$ to ${\bf c}(t_{i+1})$. When we perform a parallel transport of a vector ${\bf w}$ along  $C_n\subset S'$, the deficit angle $\Omega_n$, i.e., the angle between ${\bf w}(0)$ and ${\bf w}(1)$, depends only on the strip itself. Namely, if we cut the strip along a line and lay it flat, $\Omega_n$ is the angle between the final and the beginning edges of the cut (Fig. \ref{GM3}). The situation is exactly the same as when considering the cone tangent to the sphere along a geographic parallel at angle $\phi$. We have to sum up the contributions to the deficit angle of all the conical pieces. Thus we have
$$\Omega_n=\sum_{i=1}^n(1-\sin\phi_i)\Delta\theta_i\ \ ,$$
where $\Delta\theta_i=\theta_{i+1}-\theta_i$. But the same deficit angle will be obtained for the piecewise-smooth curve $C_n'\subset S^2$ which is obtained by connecting each point ${\bf N}_i$ on the sphere to the next point ${\bf N}_{i+1}$ by first moving along the parallel (constant $\phi$), then along the meridian (constant $\theta$). Taking the limit $n\rightarrow\infty$, the deficit angle $\Omega_n$ will approach $\Omega$ along $C\subset S$ and at the same time will approach $\Omega'$ along $C'\subset S^2$. This completes the proof.
\end{proof}
\begin{figure}[h]
\centering
\includegraphics[width=80mm]{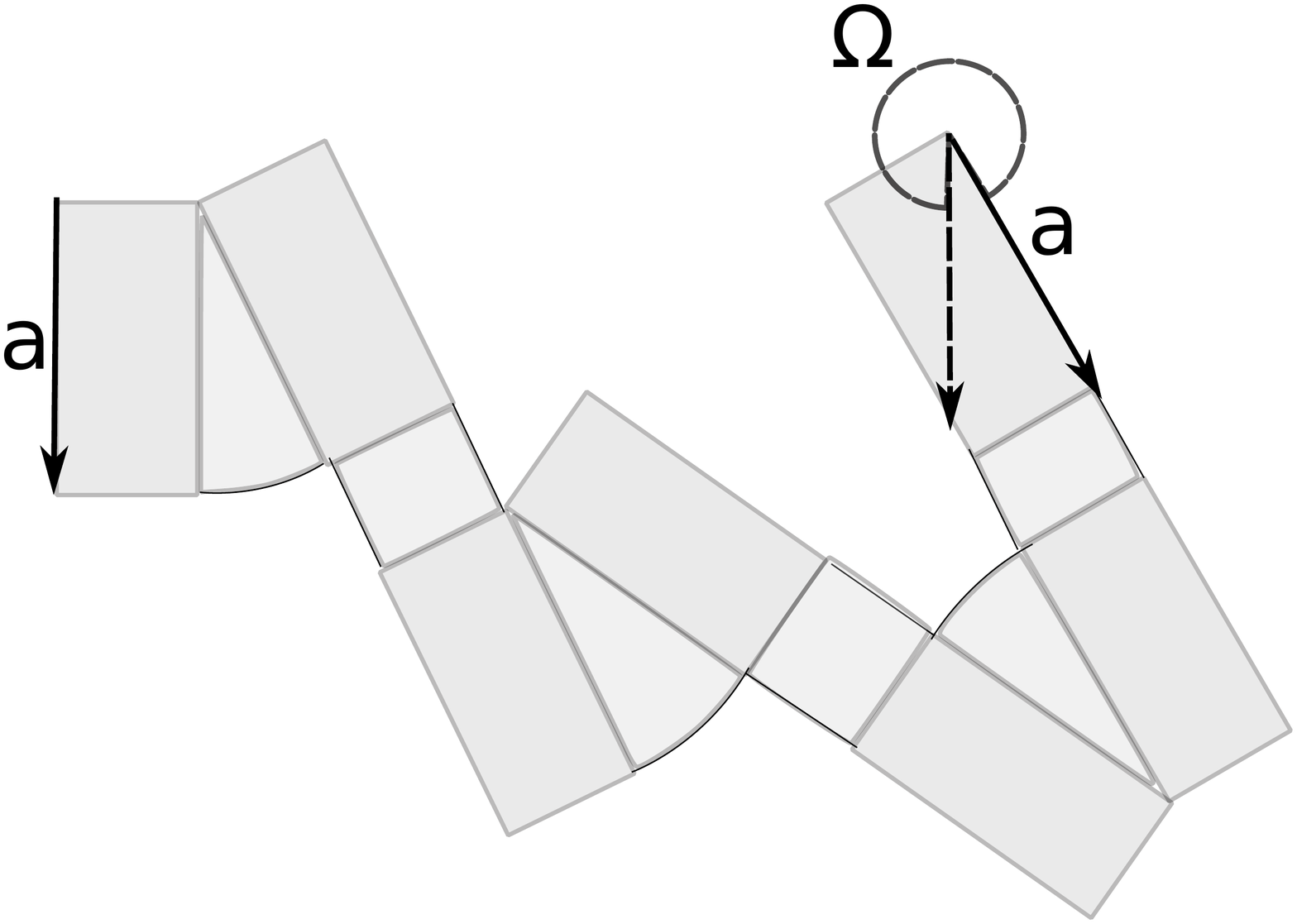}
\caption{A Gaussian-flat approximating strip, cut and laid flat \label{GM3}}
\end{figure}
If the Gauss map $G$ was one-to-one, Equation (\ref{Eq2}) would follow immediately from Proposition \ref{Prop1} and Equation (\ref{Eq5}) by a simple change of variables. Namely, if you take a simple closed contractible curve $C\subset S$ and $\sigma\subset S$ is the surface surrounded by it, $G$ will map $\sigma$ to $\sigma'\subset S^2$ and $C$ to $C'\subset S^2$, which will be the boundary of $\sigma'$. Then, by Proposition 1, the deficit angle $\Omega$ along $C$ is equal to the deficit angle $\Omega'$ along $C'$, which by Equation (\ref{Eq5}) is the surface integral over $\sigma'$ of the function 1 and this would be equal to the integral over $\sigma$ of the Jacobian of $G$. Thus if we set $K$, the {\it Gaussian curvature}, to be the Jacobian of the Gauss map (in some local coordinates on $S$ and $S^2$) we obtain exactly Equation (\ref{Eq2}). Note that  $G$ reverses the orientation if and only if $K$ is negative and the correct application of the change of variables requires that we take $K$ and not $|K|$ as the integrand. \par
The Inverse Function Theorem applied to the map $G$ says that it is one-to-one in a (small enough) neighborhood of any point for which $K\ne 0$. The surface $\sigma$ surrounded by $C$ (it is topologically like a closed disk) will be a union of, possibly countably many, open (in the relative topology on $\sigma$) simply-connected sets where $K>0$, open simply-connected sets where $K<0$, and connected closed sets where $K=0$. The boundaries of these open sets will be (piecewise smooth) curves which we can orient positively. The deficit angle $\Omega$ along an oriented curve $C$ has the properties of a line integral -- if you split the curve into pieces the total angle is the sum of the contributions along the pieces and if you reverse the orientation of $C$ the deficit angle changes its sign. 
Thus by a standard technique, we replace the deficit angle calculated along $C$ by a sum of angles calculated along boundaries of interior regions . The contribution of curves in the interior of $\sigma$ will cancel since each such curve participates twice with opposite orientation. We see that it is enough to consider curves $C$ which surround regions where $K>0$ everywhere except possibly on the boundary, or $K<0$ everywhere except possibly on the boundary, or regions (if there are such) where $K=0$. Notice that this can be performed one step at a time (see Fig. \ref{Reg}  for an illustration) -- we replace the original curve $C$ by two closed curves $C_1$ and  $C_2$, where $C_1$ surrounds a single region as above and $C_2$ surrounds all the rest. The deficit angle will be given by a possibly  infinite convergent sum of contributions for which Equation (2) holds. 
\begin{figure}[h]
\centering
\includegraphics[width=60mm]{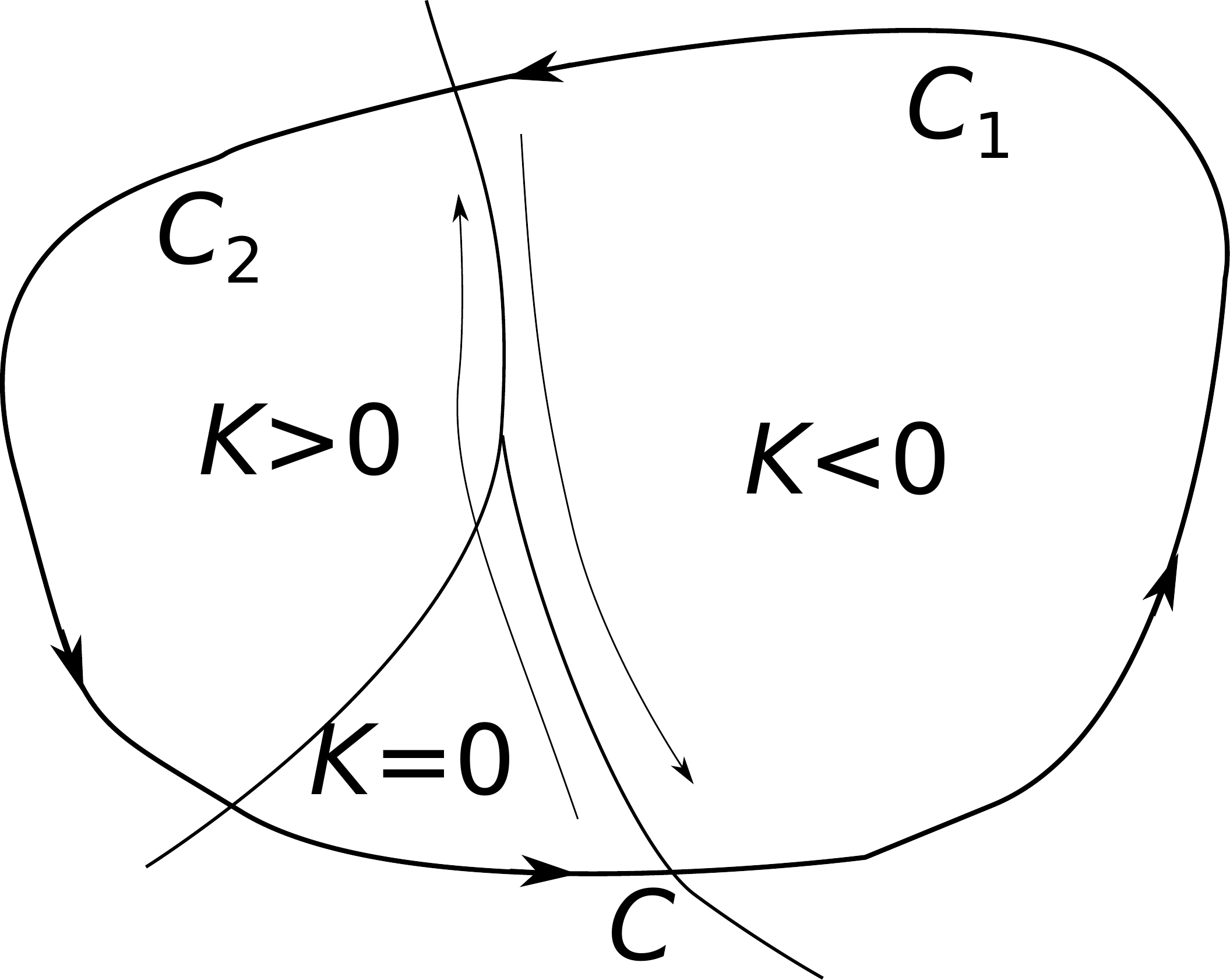}
\caption{Breaking the region surrounded by $C$ into regions of definite sign of the curvature \label{Reg}}
\end{figure}\par
Before we proceed, it is helpful to derive a convenient (and quite familiar) explicit expression for $K$ as the Jacobian of $G$. Recall that the Jacobian of a differentiable map from a (two-dimensional) manifold to another manifold of the same dimension can be viewed as the determinant of the matrix, corresponding to the differential of that map. We can write symbolically
$$K={\rm det}(dG)$$
 The differential is a linear map from the tangent space, at some point, of the first manifold, to the tangent space at the image point, of the second manifold. Namely, for a tangent vector ${\bf v}$ at some point ${\bf x}_0\in S$ we choose a smooth path ${\bf c}(t)\in S$ with ${\bf c}(0)={\bf x}_0$ and ${d{\bf c}\over dt}(0)={\bf v}$. Then the differential of $G$ at ${\bf x}_0$ is defined by
$$dG_{{\bf x}_0}({\bf v}):={d\over dt}\,(G\circ {\bf c})(0)\ \ $$
and the image is a tangent vector to $S^2$ at the point $G({\bf x}_0)$.
The Gauss map $G$ is specific in that the two tangent spaces literally coincide - the vectors, tangent to $S$ at some point ${\bf x}_0$ are orthogonal (as vectors in $\R3$) to the normal vector ${\bf N}({\bf x}_0)$ and therefore they are tangent to the sphere $S^2$ at the point ${\bf N}({\bf x}_0)$. Thus we get a linear map (operator) from $\bBR^2$ into itself:
$$W_{{\bf x}_0}\,{\bf v}:=-dG_{{\bf x}_0}({\bf v})\ \ ,$$
called the {\it shape operator} or {\it Weingarten map}. (The minus sign is a question of convention.)
The corresponding bilinear form on the tangent space at ${\bf x}_0$ , defined by the formula
$$({\bf v},{\bf w})_{_{\rm II}}:=(W_{{\bf x}_0}\,{\bf v})\cdot {\bf w}$$
is known as the {\it second fundamental form} of the surface. This form is symmetric (or equivalently, the shape operator is self-adjoint). It is enough to check symmetry for vectors forming a basis.
 Let ${\bf r}:D\subset \bBR^2\rightarrow S$ be a local parametrization of $S$ and let $(u,v)$ be the local coordinates.  The two tangent vectors ${\bf r}_u:={\partial {\bf r}\over\partial u}$ and ${\bf r}_v:={\partial {\bf r}\over\partial v}$ are linearly independent and therefore give a (not necessarily orthonormal) basis for the tangent space at any point $(u_0,v_0)$. The unit normal vector at this point is (suppressing further in the notations the dependence on the point)
$${\bf N}={{\bf r}_u\times {\bf r}_v\over \|{\bf r}_u\times {\bf r}_v\|}$$ 
and, using the definition of the differential and properties of the triple product of vectors, we calculate
$$({\bf r}_u,{\bf r}_v)_{_{\rm II}}=(W\,{\bf r}_u)\cdot{\bf r}_v=-(dG({\bf r}_u))\cdot{\bf r}_v=-{\partial{\bf N}\over\partial u}\cdot{\bf r}_v={\bf r}_{uv}\cdot{\bf N}=({\bf r}_v,{\bf r}_u)_{_{\rm II}}\ \ ,$$
where ${\bf r}_{uv}:={\partial{\bf r}_u\over\partial v}={\partial^2 {\bf r}\over\partial u\partial v}={\bf r}_{vu}$. Analogously we have
 $({\bf r}_u,{\bf r}_u)_{_{\rm II}}={\bf r}_{uu}\cdot{\bf N}$ and $({\bf r}_v,{\bf r}_v)_{_{\rm II}}={\bf r}_{vv}\cdot{\bf N}$. We obtain a symmetric $2\times 2$ matrix, expressing the second fundamental form in the basis $\{{\bf r}_u,{\bf r}_v\}$. If the basis is orthonormal the matrix will coincide with the matrix corresponding to the operator $W$ and the Gaussian curvature $K$ will be given by its determinant.The two eigenvalues $k_1$ and $k_2$ are called the {\it principal curvatures} at the point. They give the curvatures of the curves on the surface along the two eigenvectors and measure the rate of rotation of a unit normal vector along either of these curves.\par
In general, let $\{ {\bf e}_1,{\bf e}_2\}$ be an orthonormal basis in the tangent space and let $A$ be the matrix with entries
$a_{ij}=({\bf e}_i,{\bf e}_j)_{_{\rm II}}$. Let ${\bf r}_1={\bf r}_u$ and ${\bf r}_2={\bf r}_v$ and $B$ be the matrix with entries $b_{ij}=({\bf r}_i,{\bf r}_j)_{_{\rm II}}$. We have $b_{11}={\bf r}_{uu}\cdot{\bf N}$, \ 
$b_{12}=b_{21}={\bf r}_{uv}\cdot{\bf N}$ and $b_{22}={\bf r}_{vv}\cdot{\bf N}$.
 Writing ${\bf r}_i=\sum_i r_{ij}{\bf e}_j$, the components $r_{ij}$ form a matrix $R$. Using bilinearity of the second fundamental form, we obtain
$$B=RAR^{\rm t}\ ,\ \ \ \  {\rm det}B=({\rm det}R)^2\,{\rm det}A\ ,\ \ \ \ g_{ij}:={\bf r}_i\cdot{\bf r}_j=(RR^{\rm t})_{ij}\ . $$
The symmetric matrix $g$ with entries $g_{ij}$ is the Riemannean metric, in the local coordinates $(u,v)$, induced by the Euclidean metric in $\R3$. We conclude that
$$K={\rm det}A={{\rm det}B\over({\rm det}R)^2}={{\rm det}B\over{\rm det}g}
={b_{11}b_{22}-(b_{12})^2\over g_{11}g_{22}-(g_{12})^2}\ .$$
\par\smallskip\noindent
\begin{prop}\label{Prop2}
The image under the Gauss map $G$ of any closed subset of the surface $\sigma$ where $K=0$  is a closed set in $S^2$ which has no interior. 
\end{prop}
\begin{proof} It is enough to consider a finite simply-connected surface $\sigma$ with boundary, having everywhere Gaussian curvature $K=0$. The points of $\sigma$ where both principal curvatures are zero are called {\it flat points}. These are precisely the points where the so-called {\it mean curvature} $K_m:=k_1+k_2$ becomes zero. They form a closed subset of $\sigma$, which we denote by $U$. Its complement $U^c$ is an open subset of $\sigma$, consisting of the non-flat points. Through each point of $U^c$ passes a unique line, which must extend to the boundary of $\sigma$ in each direction, as we shall show (see. e.g., \cite{Massey}). Indeed, for each non-flat point we must have either $k_1=0,\ k_2\ne 0$ or $k_2=0,\ k_1\ne 0$ and we can put on $U^c$ a smooth vector field given by the eigenvectors in the "flat" direction. The integral curve $C$ of this vector field is the desired line. Notice that from the definition of the Gauss map this means that the normal vector ${\bf N}$ is constant along $C$  and therefore the tangent planes at all its points coincide, i.e., there is a common tangent plane touching $\sigma$ along the whole curve $C$. For example a plane can touch the torus along two circles - one on top, the other at the bottom. The difference between this and our situation is that the curvature of the torus is zero only along these two circles. To see that in our case $C$ is actually a line, choose a nearby integral curve $C'$. The tangent planes at $C$ and $C'$ are definitely different if we choose $C'$ to be "close enough" but different from $C$. These two planes intersect along a line. In the limit, when $C'$ approaches $C$, this line will approach both $C$ and $C'$. There exists on $U^c$ a second vector field, orthogonal to the one above - at each point of $U^c$ choose (continuously) a unit tangent vector given by the second eigenvector (in the non-flat direction) of the shape operator $W$. The integral curves of this second vector field are orthogonal to the lines constructed earlier. For an arbitrary point ${\bf x}_0\in U^c$ construct an open "trapeze" around it  by first taking the path along the non-flat direction ${\bf l}(t),\ t\in (-\epsilon, \epsilon)$ with ${\bf l}(0)={\bf x}_0$, ($t$ being the arc length), then for each  ${\bf l}(t)$ taking the unique (parametrized by arc length $s$) line ${\bf l}(s,t), s\in (-\delta,\delta)$ with ${\bf l}(0,t)={\bf l}(t)$. By varying $t$ we obtain a family of lines in $U^c$ which have the property that  each of them has a common tangent plane. This implies that for any fixed $t$ the velocity vectors ${\bf v}(s,t):={d{\bf l}\over dt}(s,t)$ are parallel for all $s$. Further, these velocities must be a linear function of $s$ or otherwise they would not produce a family of lines (See Figure. \ref{Con}). (The coefficients of this linear function depend in general on $t$.) As a consequence the arc length of the curve traced by ${\bf l}(s,t), t\in [t_1,t_2]$, $s$ fixed, is a linear function of $s$. Indeed, we can write ${\bf v}(s,t)={\bf v}_0(t)(\alpha(t)+s\beta(t))$ and assume that $\alpha(t)+s\beta(t)\ge 0$. 
\begin{figure}[h]
\centering
\includegraphics[width=60mm]{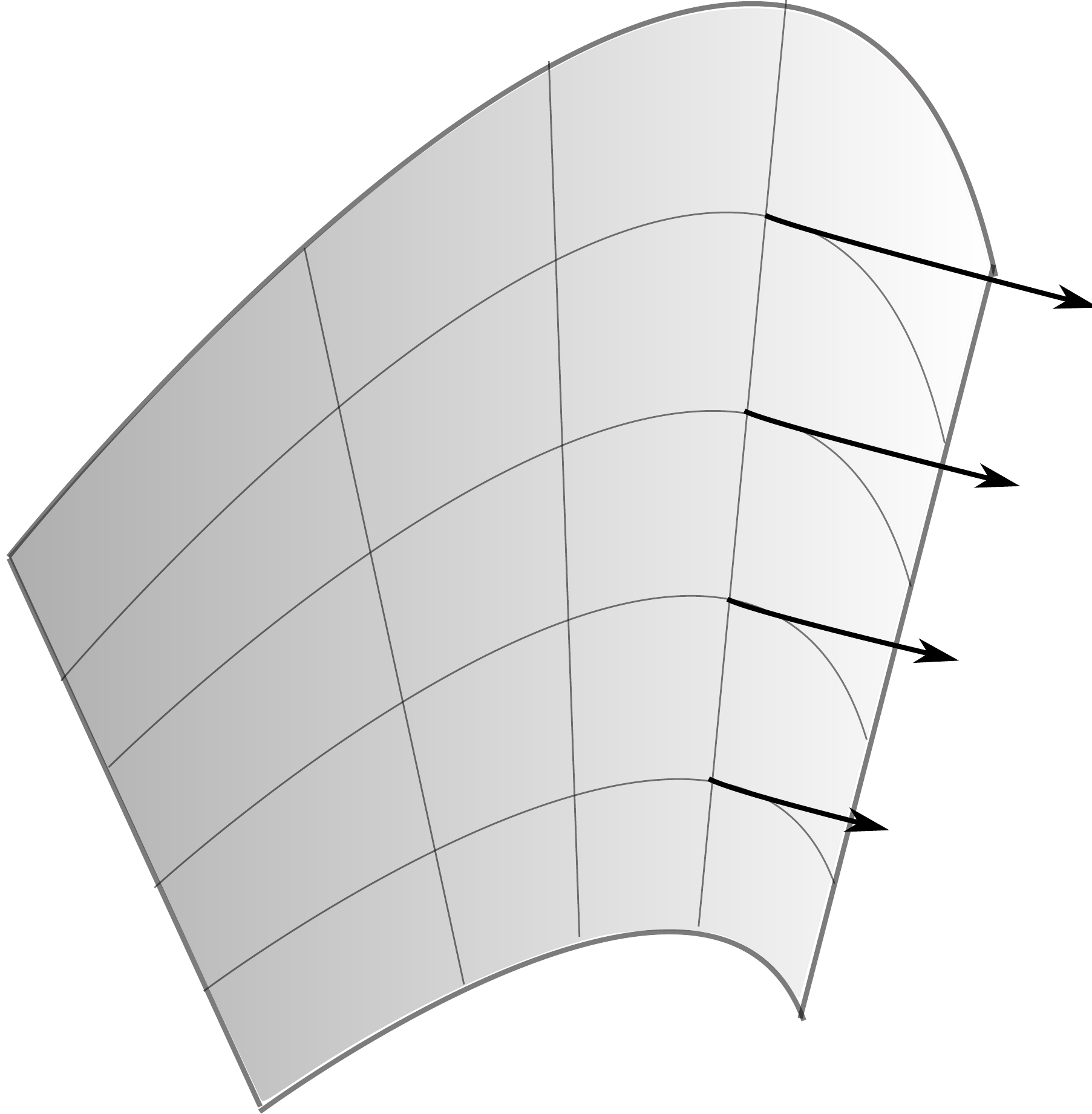}
\caption{A conical surface is a generic surface with $K=0$ \label{Con}}
\end{figure}
Then
$$L_s(t_1,t_2)=\int_{t_1}^{t_2} \|{\bf v}(s,t)\|\,dt=\int_{t_1}^{t_2}\alpha(t)\|{\bf v}_0(t)\|\,dt+s\int_{t_1}^{t_2}\beta(t)\|{\bf v}_0(t)\|\,dt\ .$$
Let $I$ be an open interval on which $\beta(t)\ne 0$. Then for any two $t_1,t_2\in I$ the arc length $L_s(t_1,t_2)$ is a non-constant  linear function of $s$ and therefore the extension of the two lines ${\bf l}(s,t_1)$ and ${\bf l}(s,t_2)$ must intersect for some $s$. In fact if we take three such lines (the extensions of) any two of them must intersect which is only possible if all of them intersect at a common point and we have a conical surface. If we have a closed interval $J$ on which $\beta(t)= 0$ then for any $t_3,t_4\in J$ the corresponding lines ${\bf l}(s,t_3)$ and ${\bf l}(s,t_4)$ will be parallel and we get a cylindrical surface.
 Now we can show an important classical result (\cite {Massey}).
\newtheorem{Lem}{Lemma}
\begin{Lem}
The reciprocal of the nonzero principal curvature is a linear function of the arc length along any line in $U^c$.
\end{Lem}
\begin{proof} Let $C$ be a line in $U^c$ and ${\bf x}_0\in C$ an arbitrary point. Construct an open "trapeze" as above parametrized by ${\bf l}(s,t)$, so that ${\bf l}(s,0)$ traces $C$ and ${\bf l}(0,0)={\bf x}_0$. Note that the parameter $s$ is arc length but the parameter $t$ is arc length only for $s=0$. Thus, according to the property shown earlier about linearity of the arc length with respect to $s$, if $(s,\tau)$ is another parametrization with $\tau$ - arc length along the non-flat directions, we have $\tau(0,t)=t$, $\tau(s,0)=0$ and therefore we must have $\tau(s,t)=t(1+s \alpha(t))$ for some coefficient function $\alpha(t)$. If $k(s,0)$ denotes the nonzero principal curvature at the point ${\bf l}(s,0)$ we can write
$$k(s,0)=\lim _{\tau\rightarrow 0}{\|{\bf N}(s,\tau)-{\bf N}(s,0)\|\over \tau}=
\lim _{t\rightarrow 0}{\|{\bf N}(0,t)-{\bf N}(0,0)\|\over t(1+s\alpha(t))}=k(0,0){1\over 1+s\alpha(0)}\ .$$
In the derivation above we used the fact that the normal vectors ${\bf N}$ are constant along any line in $U^c$ , i.e for points satisfying $t={\rm const}$.
\end{proof}
Returning to the proof of Proposition \ref{Prop2}, we see that any  line $C$ in $U^c$  cannot have as endpoint a point in $U$ since at the latter both principal curvatures must become zero but the nonzero principal curvature along $C$ approaches zero at infinity. Thus any line $C$ in $U^c$ must extend in both directions to the boundary of $\sigma$. Under the Gauss map the whole line collapses to a point and a connected component of $U^c$ will collapse to a connected curve. Since $U$ consists of the closed set of flat points in $\sigma$, it follows that any connected component of $U$ will be mapped by the Gauss map to a single point. This completes the proof. (Proposition 2 is a special case of Sard's theorem.)
\end{proof}
We have seen that a {\it developable surface} in $\R3$ can be defined in one of several equivalent ways:\\
1. A surface with $K=0$\\
2. A surface consisting of a family of lines in such a way that a tangent plane to the surface is tangent along a whole line\\
3. A surface obtained by moving a line in $\R3$ in such a way that at each moment the velocity of each point on the line is a linear function of the arc-length along this line.\par
Putting together the results of Propositions \ref{Prop1} and \ref{Prop2} we arrive at the result stated in Equation (\ref{Eq2}). Indeed, if we have a simple closed curve $C\subset S$ surrounding a region where $K>0$ (except possibly on the boundary) or a region where $K<0$ (except possibly on the boundary) , we apply directly Proposition \ref{Prop1} and the change of variables formula with the Jacobian being equal to   $K$. Note that when 
$K<0$ the Gauss map reverses orientation and so if $C$ is positively oriented its image $C'$ under $G$ will be negatively oriented and the area surrounded $C'$ should be treated as negative. In the case when $C$ surrounds a region with $K=0$ the formula still holds as the region surrounded by $C$ is mapped by the Gauss map to a closed set with no interior. Therefore we have shown
\begin{prop}\label{Prop3}
The deficit angle $\Omega$ along a simple positively oriented closed curve $C\subset S$ is equal to the integral over the surface surrounded by $C$ of the Gaussian curvature $K$.
\end{prop}
The proof of the theorem of Gauss--Bonnet is now fairly simple. Choose a triangulation of the closed oriented surface $S$ with $F$ faces (triangles), $E$ edges and $V$ vertices. Let  $\Omega_i$ be the deficit angle along the $i$th triangle $\sigma_i$. We have 
\begin{equation}\label{Eq6}
\int\hskip-6pt\int_{S} K\,dA=\sum_{i=1}^F \int\hskip-6pt\int_{\sigma_i} K\,dA=\sum_{i=1}^F \Omega_i\ .
\end{equation}
The deficit angle $\Omega_i$ is the difference between $2\pi$ and the total angle or rotation of the velocity vector ${\bf v}$ relative to the vector ${\bf w}$ which was transported parallel to itself along the (positively oriented) boundary of $\sigma_i$ . Therefore we can write
\begin{equation}\label{Eq7}
\Omega_i=2\pi - \sum_{k=1}^3 \alpha_{ik} - \sum_{k=1}^3 (\pi - \gamma_{ik})=
 \sum_{k=1}^3  \gamma_{ik}-\pi  - \sum_{k=1}^3 \alpha_{ik}\ .
\end{equation}
In the last formula $\alpha_{ik}$ denotes the angle of rotation of ${\bf v}$ relative to ${\bf w}$ along the $k$th edge of the triangle $\sigma_i$. Note that the sign of $\alpha_{ik}$ depends on the orientation of the edge. The angle $(\pi - \gamma_{ik})$ is the angle of rotation of ${\bf v}$ when it moves across the $k$th vertex of $\sigma_i$. If the triangle $\sigma_i$ is a geodesic triangle, i.e. its edges are geodesic curves and thus the quantities $\alpha_{ik}$ are all zero, we obtain from the last two equations yet another well-known characterization of the Gaussian curvature:
\begin{equation}\label{Eq8}
\int\hskip-6pt\int_{\sigma_i} K\,dA= \sum_{k=1}^3  \gamma_{ik}-\pi\ ,
\end{equation}
i.e., the integral of the Gaussian curvature over a geodesic triangle measures the deviation of the sum of the angles of the triangle from $\pi$.\par
From Equations (\ref{Eq6}) and (\ref{Eq7}) we infer
\begin{equation}
\int\hskip-6pt\int_{S} K\,dA= \sum_{i=1}^F\left( \sum_{k=1}^3  \gamma_{ik}-\pi  - \sum_{k=1}^3 \alpha_{ik}
\right)=\sum_{i=1}^F\sum_{k=1}^3  \gamma_{ik}-\pi F - \sum_{i=1}^F\sum_{k=1}^3 \alpha_{ik}\ .
\end{equation}
The first double sum in the right-hand side is the sum of all angles of all triangles and it is obviously equal to $2\pi V$, since on a smooth surface the sum of all angles at a given vertex is $2\pi$. The last double sum is zero because it contains  angles of rotation of ${\bf v}$ relative to ${\bf w}$ along edges and each edge enters twice in the sum with opposite orientations. 
The same argument as the one used in Equation (\ref{Eq3}) gives the final result:
\begin{equation}
\int\hskip-6pt\int_{S} K\,dA=2\pi(V-{F\over 2})=2\pi (V-E+F)=2\pi\chi(S)\ .
\end{equation}
Thus we have proven
\newtheorem*{theo}{Gauss--Bonnet Theorem}
\begin{theo}
The integral of the Gaussian curvature over a closed orientable surface is equal to the Euler characteristic of this surface multiplied by $2\pi$.
\end{theo}

\end{document}